\documentclass[12pt,letter,reqno]{amsart}
\usepackage[english]{babel}
\usepackage[utf8]{inputenc}
\usepackage{amsmath,amsfonts,amssymb,amsthm}
\usepackage[margin=1in]{geometry}
\usepackage{hyperref}
\usepackage{graphicx}
\usepackage{xfrac}
\usepackage{color}
\usepackage{tikz,tikz-cd}
\usepackage{float}
\usepackage{caption, subcaption}
\usepackage{comment}

\usepackage{scalerel,stackengine}
\stackMath
\newcommand\reallywidehat[1]{%
\savestack{\tmpbox}{\stretchto{%
  \scaleto{%
    \scalerel*[\widthof{\ensuremath{#1}}]{\kern-.6pt\bigwedge\kern-.6pt}%
    {\rule[-\textheight/2]{1ex}{\textheight}}
  }{\textheight}%
}{0.5ex}}%
\stackon[1pt]{#1}{\tmpbox}%
}

\newtheorem{thm}{Theorem}[section]
\newtheorem{lem}[thm]{Lemma}
\newtheorem{cor}[thm]{Corollary}

\newtheorem{prop}[thm]{Proposition}
\theoremstyle{definition}

\newtheorem{exam}[thm]{Example}
\newtheorem{ques}[thm]{Question}
\theoremstyle{remark}
\newtheorem{rem}[thm]{Remark}

\newcommand{\conv}{\mathrm{conv}}

\newcommand{\commentout}[1]{}

\newcommand{\conep}{\mathrm{C}}

\DeclareMathOperator{\cone}{cone}

\DeclareMathOperator{\aff}{aff}

\newcommand{\N}{\mathbb{N}}
\newcommand{\Q}{\mathbb{Q}}

\newcommand{\R}{\mathbb{R}}

\newcommand{\Te}{\mathbb{T}}

\newcommand{\Z}{\mathbb{Z}}

\renewcommand{\phi}{\varphi}
\renewcommand{\rho}{\varrho}
\renewcommand{\epsilon}{\varepsilon}

\DeclareMathOperator{\ehr}{ehr}
\DeclareMathOperator{\Ehr}{Ehr}
\DeclareMathOperator{\Pyr}{Pyr}

\begin{document}

\title[Weighted Ehrhart Theory]{Weighted Ehrhart Theory: \\Extending Stanley's nonnegativity theorem} 

\author[E. Bajo]{Esme Bajo}
\address{University of California, Berkeley, CA, USA}
\email{esme@berkeley.edu}

\author[R. Davis]{Robert Davis}
\address{
	Colgate University \\
	Hamilton, NY, USA }
\email{rdavis@colgate.edu}

\author[J. A. De Loera]{Jes\'us A. De Loera}
\address{University of California, Davis, CA, USA}
\email{deloera@math.udavis.edu}

\author[A. Garber]{Alexey Garber}
\address{
University of Texas Rio Grande Valley, 
Brownsville, TX, USA}
\email{alexey.garber@utrgv.edu}

\author[S. Garzon Mora]{Sof\'ia Garz\'on Mora}
\address{Freie Universität Berlin,  Berlin, Germany}
\email{sofia.garzon.mora@fu-berlin.de}

\author[K. Jochemko]{Katharina Jochemko}
\address{KTH Royal Institute of Technology, Stockholm, Sweden}
\email{jochemko@kth.se}

\author[J. Yu]{Josephine Yu}
\address{Georgia Institute of Technology, Atlanta, GA, USA}
\email{jyu@math.gatech.edu}

\keywords{
    lattice polytopes; Ehrhart theory; weighted enumeration; nonnegativity }
    \subjclass[2020]{52B20; 05A15; 52B45}

\date{\today}

\begin{abstract} 
We generalize R.\ P.\ Stanley's celebrated theorem that the $h^\ast$-polynomial of the Ehrhart series of a rational polytope has nonnegative coefficients and is monotone under containment of polytopes.  We show that these results continue to hold for \emph{weighted} Ehrhart series where lattice points are counted with polynomial weights, as long as the weights are homogeneous polynomials decomposable as sums of products of linear forms that are nonnegative on the polytope. We also show  nonnegativity of the $h^\ast$-polynomial as 
a real-valued function for  a larger family of weights. 

We explore the case when the weight function is the square of a single (arbitrary) linear form. We show stronger results for two-dimensional convex lattice polygons and give concrete examples showing tightness of the hypotheses.
 As an application, we construct a counterexample to a conjecture by Berg, Jochemko, and Silverstein on Ehrhart tensor polynomials.
\end{abstract}

\maketitle

\section{Introduction}\label{sec:intro} 

Let $P\subseteq\R^d$ be a rational convex polytope, that is, a polytope with vertices in $\mathbb{Q}^d$, and let $w: \R^d \to \R$ be a polynomial function, often called a  \emph{weight function}.
A computational problem arising throughout the mathematical sciences is to compute, or at least estimate, the sum of the values $w(x) := w(x_1,\dots,x_d)$ over the set of integer points belonging to $P$, namely
\[
	\ehr(P,w)=\sum_ {x\in P \cap \Z^d} w(x).
 \]

 Weighted sums of the above type are also a classical topic in convex discrete geometry where they have been studied for a long time under the name 
 \emph{polynomial valuations}~\cite{McMullen,PukhlikovKhovanskii,Alesker,BL17}. They appear 
 in the work of Brion and Vergne \cite{BrionVergne1}, who used weighting in the 
 context of Euler-Maclaurin formulas. Other ideas of what it means to be weighted 
 have been proposed later on, for instance, by Chapoton \cite{chapoton_2016}, who developed a related $q$-theory for the case when $w(x)$ is a linear form, also by Stapledon~\cite{Stapledon}, who explored a grading with piece-wise linear functions, and by Ludwig and Silverstein~\cite{Ludwig}, who introduced and studied Ehrhart tensor polynomials based on discrete moment tensors.

Important applications of such weighted problems appear, for instance, in
enumerative combinatorics \cite{andrewspaulereise,quasiweightsfpsac2023}, 
statistics \cite{DG, Chenetalstats}, non-linear optimization
\cite{deloera-hemmecke-koeppe-weismantel:intpoly-fixeddim}, and weighted 
lattice point sums, which have played a key role in the computation of volumes 
and integrals over polytopes \cite{BBDKV08}.
{\sloppy

}

The sum of the weighted integer points in the $n$-th dilate of the 
rational polytope $P$ for nonnegative integers $n\in \N$ is given by 
the \emph{weighted Ehrhart function} $\ehr(nP,w)$. The main object of 
this article is the generating function
  \[
    \Ehr(P, w; t) = \sum_{n\geq 0} \ehr(nP,w) t^n
  \]
  called the \emph{weighted Ehrhart series}. The fact that the weighted Ehrhart series is a rational function has been known for a long time, e.g., it has been used in computational software for at least ten years (\cite{LattE,BrunsSoger2015}). For our purposes we see in Proposition~\ref{prop:rational} why $\Ehr(P, w; t)$ is a rational function of the form 
  \[
    	\Ehr(P, w; t) = \frac{h_{P, w}^*(t)}{(1-t^q)^{r+m+1}}
      \]
      whenever $P$ is an $r$-dimensional rational polytope; here $m=\deg(w)$, $h^*_{P,w}(t)$ is a polynomial of degree at most $q(r+m+1)-1$, and $q$ denotes the smallest positive integer such that $qP$ has vertices in $\Z^d$, called the \emph{denominator} of $P$. 
      
      We say that the empty polytope has denominator $1$. We call $h_{P,w}^*(t)$
the \emph{weighted $h^*$-polynomial of $P$} and its list of
coefficients the \emph{weighted $h^*$-vector of $P$} with respect to the weight $w$.

From the rationality of $\Ehr(P,w;t)$, it follows that the weighted Ehrhart function $\ehr(nP,w)$ is a
quasi-polynomial in $n$, that is, it has the form
{\sloppy

}
\[
	\ehr(nP,w)= \sum_{i=0}^{d+ m} E_i(n) n^i
\]
where the coefficients $E_i : \N \rightarrow \R$ are
periodic functions with periods dividing the denominator of $P$.
The leading coefficient of the $h^*$-polynomial is equal to the integral of the weight $w$ over the polytope $P$; these integrals were studied in \cite{Barvinok-1991},
\cite{Barvinok-1992} and \cite{BBDKV08}. If $w=1$, that is, if $\ehr(nP,w)=|nP \cap \mathbb{Z}^d|$, then we recover the classical Ehrhart theory counting lattice points in dilates of polytopes. Even in this
case, it is an NP-hard problem to compute all of the coefficients $E_i$. See
\cite{barvinokzurichbook, beckrobins} for excellent introductions to
this topic. 
{\sloppy

}

In the classical case of $w=1$, a fundamental theorem by Richard P. Stanley, often called \emph{Stanley's nonnegativity theorem}, states that the $h^\ast$-polynomial of any rational polytope  has only nonnegative integer coefficients~\cite{StanleyDecompositions}. Even stronger, for rational 
polytopes $P$ and $Q$ such that $P\subseteq Q$, Stanley proved $h^\ast _{P,1}(t)\preceq h^\ast _{Q,1}(t)$ where $\preceq$ denotes coefficient-wise comparison. This last property has been known as \emph{$h^\ast$ monotonicity property}. For details and proofs see e.g., \cite{beck-sottile:irrational,StanleyDecompositions,StanleyMonotonicity}.

Positivity and nonnegativity of coefficients is important in algebraic combinatorics (see e.g., \cite{positivityStanley2000} and its references), 
but we must stress that one nice aspect of our results is they connect to the nonnegativity of the associated $h^\ast$-polynomial as real-valued functions. This is a topic that goes back to the work on real algebraic geometry by Hilbert, P\'olya, Artin and others (see \cite{bookPositivePolynomials2001,Marshall}), and it has seen renewed activity in the classical methods of moments, real algebraic geometry, and sums of squares decompositions for polynomials because it provides a natural approach for optimization algorithms (see \cite{Marshall,BlekhermanParriloThomasbook}). 

Motivated by this prior work and context, our article discusses \emph{the nonnegativity and monotonicity properties of the coefficients of weighted $h^\ast$-polynomials, as well as their non-negativity as real-valued functions.} 

\subsection{Our Contributions}
In contrast to its classical counterpart, the weighted $h^\ast$-polynomial may have negative coefficients, even when the weight function is nonnegative over the 
polytope and all of its nonnegative dilates. For example, when $P$ is the line segment $[0,1]\subseteq\R$, one can calculate that
{\sloppy

}
\[
\Ehr(P,1; t) = \frac{1}{(1-t)^2} \quad  \text{ and } \quad \Ehr(P,x^2; t) = \frac{t^2+t}{(1-t)^4},
  \]
  and so their sum is
  \[ \Ehr(P,x^2+1; t) = \frac{2t^2-t+1}{(1-t)^4}.
    \]
    
As can be seen in this simple example, adding Ehrhart series corresponding to
 weights of different degrees may introduce negative coefficients to the $h^\ast$-polynomial since the rational functions have different denominators. 
 We therefore focus on homogeneous polynomials as weight functions. For an investigation of how to deal with more general weights see \cite{quasiweightsfpsac2023}.
 
We now consider the following, slightly more general setup, where the 
weight function $w$ may depend not only on the coordinates of the points 
$nP \cap \Z^d$ but also on the scaling factor $n$.  For any rational 
polytope $P\subseteq \R^d$, the \emph{cone over $P$} (or \emph{homogenization} 
of $P$) is the rational polyhedral cone in $\R^{d+1}$ defined as
\[
    \conep(P) := \cone (P \times \{1\}) = \{ c (p,1) \mid c \geq 0, p \in P\}.
\]

For any polynomial $w$ in $d+1$ variables we consider the weighted Ehrhart series
\[
\Ehr(P,w;t) = \sum_{x \in
  \conep(P) \cap \Z^{d+1}} w(x) t^{x_{d+1}}.
\]
Let $\conep(P)^*$ be the cone consisting of the linear
functions on $\R^{d+1}$ which are nonnegative on $\conep(P)$. 
If the cone $\conep(P)$ is defined by linear inequalities $\ell_1 \geq 0,\dots,\ell_m \geq 0$, then $\conep(P)^*$ is a polyhedral cone generated by nonnegative linear combinations of $\ell_1,\dots,\ell_m$.
We focus on the following two families of polynomials in $d+1$ variables as
weights functions:
\begin{itemize}
    \item [(i)] the semiring $R_P$ consisting of sums of
products of linear forms in $\conep(P)^*$.   Each element
of $R_P$ has the form $c_1 \ell^{a_1} + \cdots + c_k
\ell^{a_k}$ where $c_1,\dots, c_k$ are positive real numbers and $\ell^{a_1},\dots,\ell^{a_k}$ are monomials in the generators 
$\ell_1,\dots,\ell_m$ of $\conep(P)^*$; and

\item[(ii)] the semiring $S_P$ consisting of sums of nonnegative products 
of linear forms on $P$. If a product of linear forms is nonnegative on $P$, 
then each of the linear forms involved is either nonnegative on all of $P$ 
or appears with an even power; otherwise the product would change sign across 
the hyperplane where the linear form vanishes.  Therefore, an element of $S_P$
has the form $s_1 \ell^{a_1} + \cdots + s_k \ell^{a_k}$ where $s_1,\dots, s_k$ 
are squares of products of any linear forms and $\ell^{a_1},\dots,\ell^{a_k}$ are monomials in the generators $\ell_1,\dots,\ell_m$ of $\conep(P)^*$.
\end{itemize}

In $R_P$ each of the linear forms involved are
nonnegative on $P$. In contrast, in $S_P$, each product is nonnegative
but the individual linear forms may have negative values in $P$.
Thus we have $R_P \subseteq S_P$.  Both semirings are
contained in the {\em preordering} generated by $\ell_1,\dots,\ell_m$ 
consisting of elements of the form $s_1 \ell^{a_1} + \cdots + s_k
\ell^{a_k}$ where $s_i$ are arbitrary squares of polynomials instead of just squares of products of linear forms.  See, for example,~\cite{Marshall}.

The main results of this article are the following. 
\theoremstyle{theorem}
\newtheorem*{mainthm}{Nonnegativity Theorem.  (Theorem \ref{thm:main})}
\begin{mainthm}
  Let $P$ be a rational polytope, $\conep(P)$ its cone, and $\conep(P)^*$ the dual cone of  linear functions on $\R^{d+1}$ which are nonnegative on $\conep(P)$. Let $R_P$ and $S_P$ be, respectively, the semirings of sums of products of linear forms in $\conep(P)^*$ and of sums of nonnegative products of linear forms on $P$.
  \begin{enumerate}
\item   If  the weight $w$ is a homogeneous element of $R_P$, then the
coefficients of $h_{P,w}^*(t)$ are nonnegative.
\item  If  the weight $w$ is a homogeneous element of $S_P$, then
  $h_{P,w}^*(t) \geq 0$ for $t \geq 0$.
  \end{enumerate}
  \end{mainthm}

As we mentioned before Stanley also showed that the classical $h^*$-polynomials satisfy a monotonicity property: for lattice polytopes $P$ and $Q$, of possibly different dimension, such that $P \subseteq Q$, we have $h^*_P(t) \preceq h^*_Q(t)$ where $\preceq$ denotes the coefficient-wise inequalities \cite{StanleyMonotonicity}. This can be seen as a generalization of the nonnegativity theorem when we set $P=\emptyset$ in which case the Ehrhart series and thus the $h^\ast$-polynomial is zero.
Now we are able to prove the following:

\theoremstyle{theorem}
\newtheorem*{thm:monotone1}{First Monotonicity Theorem.  (Theorem \ref{thm:monotone1})}
\begin{thm:monotone1}
 Let $P,Q\subseteq\R^d$ be rational polytopes, $P\subseteq Q$, and let $g$ be a common multiple of the denominators $\delta (P)$ of $P$ and $\delta (Q)$ of $Q$. Then, for all weights $w\in R_Q$,
    \[
(1+t^{\delta(P)}+\cdots +t^{g-\delta (P)})^{\dim P+m+1}h^\ast _{P,w}(t) \preceq (1+t^{\delta(Q)}+\cdots +t^{g-\delta (Q)})^{\dim Q+m+1}h^\ast _{Q,w}(t) \, .
\]
In particular, if $P \subseteq Q$ are polytopes with the same denominator, then taking $g=\delta(P)=\delta(Q)$ gives
    \begin{equation}
h^\ast _{P,w}(t) \preceq h^\ast _{Q,w}(t).
    \end{equation}
\end{thm:monotone1}

\theoremstyle{theorem}
\newtheorem*{thm:monotone2}{Second Monotonicity Theorem.  (Theorem \ref{thm:monotone2})}
\begin{thm:monotone2}
 Let $P,Q\subseteq\R^d$ be rational polytopes of the same dimension $D=\dim P=\dim Q$, $P\subseteq Q$, and let $g$ be a common multiple of the denominators $\delta (P)$ of $P$ and $\delta (Q)$ of $Q$. Then, for all weights $w\in S_Q$,
    \[
(1+t^{\delta(P)}+\cdots +t^{g-\delta (P)})^{D+m+1}h^\ast _{P,w}(t) \leq (1+t^{\delta(Q)}+\cdots +t^{g-\delta (Q)})^{D+m+1}h^\ast _{Q,w}(t) \, 
\]
for all $t\geq 0$.
    In particular, if $P \subseteq Q$ are polytopes with the same denominator and dimension, then taking $g=\delta(P)=\delta(Q)$ gives
    \begin{equation}
h^\ast _{P,w}(t) \leq h^\ast _{Q,w}(t)   \text{ for all }t\geq 0.
    \end{equation}
\end{thm:monotone2}

 We wish to emphasize that while Theorem \ref{thm:main} is a generalization of Stanley's nonnegativity theorem, Theorem \ref{thm:monotone1} is closer in spirit to \emph{P\'olya's theorem  on positive polynomials} which says that if a homogeneous polynomial  $f\in {\mathbb R}[X_1, \dots, X_n]$ is strictly positive on the standard simplex 
 \[
    \Delta_n := \{(x_1, \dots, x_n)\in {\mathbb R}^n \mid x_1,\dots,x_n\geq 0, x_1+\dots+x_n=1\},
\]
then for sufficiently large $N$, all of the nonzero coefficients of $(X_1+\dots+X_n)^Nf(X_1, \dots,X_n)$ are strictly positive. Note also, the semiring $R_P$ is a homogenized version of the semiring appearing in \emph{Handelman's theorem} \cite{handelman1988} which says that all polynomials strictly positive on a polytope $P$ lie in the semiring generated by the linear forms which are nonnegative on the polytope. We remark that all homogeneous polynomials are sums of (unrestricted) products of linear forms and it is an important problem to find such decompositions (see \cite{alexanderhirschowitz,MourrainOneto2020,ReznickMonograph1992} and references therein).
Thus our restriction to $R_P$ and $S_P$ is a natural approach to understanding nonnegativity and bringing us close to the best possible result.

To study the limitations of our results we focus on the case when the weight is given by a single arbitrary linear form. In this case we strengthen our results for two dimensional lattice polygons.

\theoremstyle{theorem}
\newtheorem*{thm:polygon}{Theorem \ref{thm:polygon}}
\begin{thm:polygon}
For every (closed) convex lattice polygon $P$ and every linear form $\ell$, the $h^*$-polynomial of $P$ with respect to $w(x)=\ell^2(x)$ has only nonnegative coefficients.
\end{thm:polygon}

In particular, this shows that the weighted $h^\ast$-polynomial of any convex lattice polygon has nonnegative coefficients, even when the linear form takes negative values on the polygon. Furthermore, we provide examples that show that this result is no longer true if the assumptions on the polytope or weight are relaxed. In particular, we construct a $20$-dimensional lattice simplex and a linear form such that the $h^\ast$-polynomial with respect to the square of the linear form has a negative coefficient (Example~\ref{ex:negative}). These results have interpretations and implications in terms of generating functions of Ehrhart tensor polynomials. In particular, the example mentioned above gives a counterexample to a conjecture of Berg, Jochemko and Silverstein~\cite[Conjecture 6.1]{Berg} on the positive semi-definiteness of $h^2$-tensor polynomials of lattice polytopes (Corollary~\ref{cor:tensorcounterexample}).

Unlike the classical results of Stanley for $w=1$, where techniques from commutative algebra can be applied since the Ehrhart series is actually the Hilbert series of a graded algebra, we do not see an obvious connection to commutative algebra methods. Instead, to prove Theorems~\ref{thm:main} and~\ref{thm:monotone1} we consider the cone homogenization of polytopes and half-open decompositions and follow a variation of the triangulation ideas first outlined by Stanley in~\cite{StanleyDecompositions}. While this methodology has been used by many authors since then \cite{BajoBeck,beck-sottile:irrational,JochemkoSanyal}, we require a careful analysis of the properties of the semirings $R_P$ and $S_P$. For this we consider multivariate generating functions for half-open cones and provide explicit combinatorial interpretations using generalized $q$-Eulerian polynomials~\cite{SteingrimssonEulerian}. The $q$-Eulerian polynomials and their relatives frequently appear in enumerative and geometric combinatorics~\cite{BrandenEulerian,BrentiWelker,BeckZonotopes,HaglundZhang}.

This article is organized as follows. In Section~\ref{sec:ehr} we give an explicit formula for the weighted multivariate generating function for half-open simplicial cones (Lemma~\ref{lem:formula}). This formula then allows us to show the rationality of the (univariate) weighted Ehrhart series (Proposition~\ref{prop:rational}) as well as the first part of Theorem~\ref{thm:main} by specialization and using half-open decompositions. The second part of Theorem~\ref{thm:main} is obtained by considering subdivisions of the polytope induced by the linear forms involved in the weight function. A more refined analysis then also allows us to prove the monotonicity Theorems~\ref{thm:monotone1} and~\ref{thm:monotone2}. In Section~\ref{sec:2d} we focus on the case when the weight function is given by a square of a single linear form and prove Theorem~\ref{thm:polygon}. We also show that the assumptions on convexity, denominator, dimension and degree are necessary by providing examples. In Section~\ref{sec:tensors} we describe the connections and implications of our results to Ehrhart tensor polynomials. In particular, we show that weighted Ehrhart polynomials can be seen as certain evaluations of Ehrhart tensor polynomials (Proposition~\ref{prop:identification}), and thus, positive semi-definiteness of $h^2$-tensor polynomials is equivalent to nonnegativity of weighted $h^\ast$-polynomials with respect to squares of linear forms (Proposition~\ref{prop:positivesemidefinite}). In particular, Example~\ref{ex:negative} disproves~\cite[Conjecture 6.1]{Berg} (Corollary~\ref{cor:tensorcounterexample}).

\section{Nonnegativity and monotonicity of weighted $h^*$-polynomials} \label{sec:ehr}
\subsection{Generating series}
Let $P\subseteq \mathbb{R}^d$ be a rational polytope of dimension $r$ with denominator $q$ and let $w : \mathbb{R}^d\rightarrow \mathbb{R}$ be a polynomial of degree $m$. In this section we will see that $\Ehr(P,w;t)$ is a rational function of the form
  \[
    	\Ehr(P, w; t) = \frac{h_{P, w}^*(t)}{(1-t^q)^{r+m+1}}\, ,
      \]
where $h_{P, w}^*(t)$ is a polynomial of degree at most $q(r+m+1)-1$.  Our main goal is to study positivity properties of the numerator polynomial. 
Our approach uses general multivariate generating series of half-open simplicial cones and specializing to obtain the univariate generating function of the homogenization $\conep(P)$ following ideas outlined in~\cite{StanleyDecompositions} but requiring careful analysis of the semirings $R_P$ and $S_P$.

For a polynomial $w(x)$ in $d$ variables, the  multivariate weighted lattice point generating function of the cone $C$ is 
$ \sum_{x \in C \cap \Z^{d}} w(x) z^{x}$ where
$z^x = z_1^{x_1}\cdots z_d^{x_d}$ is a monomial in $d$ variables.  We will now show that this generating function is a rational function and give an explicit formula when the weight is a product of linear forms. 

Our expression uses the following parametrized generalization of Eulerian polynomials. 
For $\lambda \in [0,1]$, let $A_d^{\lambda}(t)$
be the polynomial defined by
\[
\sum_{n \geq 0} (n+\lambda)^dt^n = \frac{A_d^{\lambda}(t)}{(1-t)^{d+1}}.
\]
If $\lambda = 1$, then this is the usual Ehrhart series of a $d$-dimensional unit cube, and $A_d^1(t)$ is the Eulerian polynomial, all of whose roots are real and nonpositive.
If $\lambda = \tfrac{1}{r}$ for some integer $r\geq 1$ then $r^dA_d^{\lambda}(t)$ equals the $r$-colored Eulerian polynomial~\cite{SteingrimssonEulerian}.
For each $\lambda \in [0,1]$ the polynomial $A_d^{\lambda}(t)$ also has only real, nonpositive roots~\cite[Theorem 4.4.4]{BrentiMemoirs}. In particular, all of its coefficients are nonnegative. We formally record this in a lemma.

\begin{lem}[{\cite[Theorem 4.4.4]{BrentiMemoirs}}]
  \label{lem:Eulerian}
For any integer $d\geq 1$ and real number $\lambda \in [0,1]$, the coefficients of $A^\lambda_d(t)$ are nonnegative.
\end{lem}

Our computations additionally use some concepts which we now introduce. For consistency, we assume that the polytopes live in the $d$-dimensional space $\R^d$ while cones live in the ambient space $\R^{d+1}$.

Let $C$ be a {\em half-open} $(r+1)$-dimensional simplicial cone in $\R^{d+1}$ generated by nonzero integer vectors $v_1,\dots, v_{r+1} \in \Z^{d+1}$ with the first $k$ facets removed where $0 \leq k \leq r+1$.  More precisely,
  \[
C = \{ c_1 v_1 + \cdots + c_{r+1} v_{r+1} \mid c_1, \dots, c_k > 0, \,
c_{k+1}, \dots, c_{r+1} \geq 0\}.
\]
Since $C$ is simplicial, every point $\alpha \in C$ can be written uniquely as
  \[
\alpha = x + s_1 v_1 + \cdots + s_{r+1} v_{r+1}
\]
where $s_1,\dots,s_{r+1}$ are nonnegative integers, and $x$ is in the
\emph{half-open parallelepiped}
\[
\Pi = \{\lambda_1v_1 + \cdots + \lambda_{r+1} v_{r+1} \mid 0 < \lambda_1, \dots, \lambda_k \leq 1, \, 0 \leq \lambda_{k+1}, \dots, \lambda_{r+1} < 1\}.
\]

We obtain the following explicit formula for the multivariate generating function of a half-open simplicial cone if the weight is a product of linear forms. Since every polynomial is a sum of product of linear forms, namely monomials, this gives a formula to compute the generating function for any polynomial weight.

\begin{prop}
\label{lem:formula}
Let $C$ be an $(r+1)$-dimensional half-open simplicial cone in $\R^{d+1}$ with generators $v_1,\hdots,v_{r+1}$ in $R_P$.  Let $w = \ell_1 \cdots \ell_m$ be a product of linear forms in $d+1$ variables. Then  
\begin{equation}
       \sum_{\alpha \in C \cap \Z^{d+1}} w(\alpha) z^{\alpha}   = \sum_{x \in \Pi \cap \Z^{d+1}} \left( z^{x} \sum_{\substack{
        I_1 \uplus \cdots \uplus I_{r+1} = [m]}} \prod_{i \in I_1}
    \ell_i(v_1)\cdots\prod_{i \in
I_{r+1}}\ell_i(v_{r+1})\prod_{j=1}^{r+1}\frac{A_{|I_j|}^{\lambda_j(x)}(z^{v_j})}{(1-z^{v_j})^{|I_j|+1}}\right)
    \end{equation}
    where $\Pi$ is the half-open parallelepiped as above and each $x \in \Pi$ is written $x = \lambda_1(x) v_1 + \cdots +
    \lambda_{r+1}(x) v_{r+1}$.  The innermost sum runs over all the ordered partitions of $[m]$ into $r+1$ (possibly empty) parts and $I_1\uplus\cdots\uplus I_{r+1}$ denotes the disjoint union of these parts.
\end{prop}

Note that when $m=0$ and the weight is constant, 
there is only the partition into empty sets where by definition the products are all $1$ (empty products) and the Eulerian polynomials are all $1$.

\begin{proof}
Using that any $\alpha\in C$ is $\alpha=x+s_1v_1+\cdots s_{r+1}v_{r+1}$ for $x\in \Pi$, the generating function is   
    \[
        \begin{aligned}
           \sum_{\alpha \in C \cap \Z^{d+1}} w(\alpha) z^{\alpha}  &= \sum_{\alpha \in C \cap \Z^{d+1}}
           \left(\prod_{i=1}^m \ell_i(\alpha)\right) z^{\alpha} \\
            &= \sum_{x \in \Pi \cap \Z^{d+1}} \sum_{s_1 \geq 0} \cdots
            \sum_{s_{r+1} \geq 0} \underbrace{\prod_{i=1}^m
              \ell_i({x+s_1v_1+\cdots+s_{r+1}v_{r+1}})}_{(*)}z^{x+
              s_1 v_1+\cdots+s_{r+1} v_{r+1}}.
        \end{aligned}
    \]
    Since $x \in \Pi$ is $x = \lambda_1(x) v_1 + \cdots +
    \lambda_{r+1}(x) v_{r+1}$, using linearity of each $\ell_i$ we can expand out
    \begin{align*}
        (*) &= \sum_{
        \substack{
        I_1 \uplus \cdots \uplus I_{r+1} = [m]}} \left[ \prod_{i \in I_1}
      \ell_i((s_1+\lambda_1(x)) v_1)\right] \cdots \left[\prod_{i \in
      I_{r+1}}\ell_i((s_{r+1}+\lambda_{r+1}(x)) v_{r+1})\right]
    \\
 &=    \sum_{\substack{
        I_1 \uplus \cdots \uplus I_{r+1} = [m]}} \left[ \prod_{i \in I_1} \ell_i(v_1)\right]\cdots \left[\prod_{i \in I_{r+1}}\ell_i(v_{r+1}) \right](s_1+\lambda_1(x))^{|I_1|}\cdots(s_{r+1}+\lambda_{r+1}(x))^{|I_{r+1}|}.
  \end{align*}
    where $i \in I_j$ represents the term $(s_j+\lambda_j(x))v_j$ being chosen from $\ell_i$ when multiplying out. 
    Placing this into our original series, we obtain
    \begin{equation}
       \resizebox{.9 \textwidth}{!}{${\displaystyle\sum_{\alpha \in C \cap \Z^{d+1}} w(\alpha) z^{\alpha}   = \sum_{x \in \Pi \cap \Z^{d+1}} z^x \sum_{\substack{
        I_1 \uplus \cdots \uplus I_{r+1} = [m]}} \prod_{i \in I_1}
    \ell_i(v_1)\cdots\prod_{i \in
      I_{r+1}}\ell_i(v_{r+1})\prod_{j=1}^{r+1}\left(\sum_{s_j \geq
        0}(s_j+\lambda_j(x))^{|I_j|}z^{s_jv_j}\right)}$.}
    \end{equation}
    For the innermost sum on the right, we can write for each $j$
    \begin{equation}
      \label{eqn:Eulerian}
        \sum_{s_j \geq 0}(s_j+\lambda_j(x))^{|I_j|}(z^{v_j})^{s_j} = \frac{A_{|I_j|}^{\lambda_j(x)}(z^{v_j})}{(1-z^{v_j})^{|I_j|+1}}.
    \end{equation}
This completes the proof.
\end{proof}

In order to show that $\Ehr(P,w;t)$ is a rational function for any rational polytope $P$ we consider partitions into half-open simplices. 
Given affinely independent vectors $u_1,\ldots, u_{r+1}\in \R^{d}$, the half-open simplex with the first $k\in \{0,1,\ldots,r+1\}$ facets removed is defined as 
\[
\Delta = \left\{\sum _{i=1}^{r+1} c_i u_i \mid  c_1,\ldots c_k >0, c_{k+1},\ldots, c_{r+1}\geq 0,\sum _{i=1}^{r+1} \lambda _i=1 \right\} \, ,
\]
and the homogenization of $\Delta$ is the half-open simplicial cone 
  \[
\conep(\Delta) = \{ c_1 v_1 + \cdots + c_{r+1} v_{r+1} \mid c_1 > 0, \dots, c_k > 0,
c_{k+1} \geq 0, \dots, c_{r+1} \geq 0\}
\]
where $v_i=(u_i,1)$ for all $i$.

Given an $r$-dimensional polytope $P$ and a triangulation, we can partition $P$ into half-open simplices in the following way. Let $q$ be a generic point in the relative interior of $P$ and let $S=\conv \{u_1,\ldots, u_{r+1}\}$ be a maximal cell in the triangulation where $\conv\{\cdot \}$ denotes the convex hull. We say that a point $p\in S$ is \emph{visible} from $q$ if $(p,q] \cap S =\emptyset$. A half-open simplex, denoted $H_q S$, is then obtained by removing all points that are visible 
from $q$, which can be seen to be equal to
\[
H_q S = \{c_1u_1+\cdots +c_{r+1}u_{r+1} \in S \mid c_i>0 \text{ for all }i\in I_q\}
\]
where $I_q=\{i\in [r+1]\mid u_i \text{ not visible from } q\}$.

The following is a special case of a result of K\"oppe and Verdoolaege~\cite{KoeppeHalfopen}.
\begin{thm}[{\cite{KoeppeHalfopen}}]
\label{thm:halfopen}
Let $P$ be a polytope, $q\in \aff P$ be a generic point and $S_1,\ldots, S_m$ be the maximal cells of a triangulation of $P$. Then
\[
P \ = \ H_q S_1\uplus H_q S_2 \uplus \cdots \uplus H_q S_m
\]
is a partition into half-open simplices.
\end{thm}
With the notation as in the previous theorem, it follows that
\begin{equation}\label{eq:halfopendecompcones}
    \conep(P) \ = \ \conep(H_q S_1)\uplus \conep(H_q S_2) \uplus \cdots \uplus \conep(H_q S_m), 
\end{equation}
that is, the homogenization $\conep(P)$ of $P$ can be partitioned into half-open simplicial cones. This, together with Proposition~\ref{lem:formula}, allows us to show rationality of $\Ehr (P,w;t)$.

\begin{prop}\label{prop:rational}
 For any rational polytope $P$ of dimension $r$ and any degree-$m$ form $w$ on $\conep(P)$, the weighted Ehrhart series is a rational function of the form
  \[
\Ehr(P, w; t) = \frac{h^*_{P,w}(t)}{(1-t^q)^{r+m+1}}
\]
where $q$ is a positive integer such that $q P$ has integer vertices and $h^\ast _{P,w} (t)$ is a polynomial of degree at most $q(r+m+1)-1$.
\end{prop}

\begin{proof}
Let $S_1,\ldots, S_m$ be the maximal cells of a triangulation of $P$ using no new vertices, that is, for all $i$, the vertex set of $S_i$ is contained in the vertex set of $P$. Let 
\[
P \ = \ H_q S_1\uplus H_q S_2 \uplus \cdots \uplus H_q S_m
\]
be a partition into half-open simplices, and let \[
\Ehr (H_q S_i, w;t)=\sum _{x\in C(H_q S_i)\cap \Z^{d+1}}w(x)t^{x_{d+1}}
\]
for all $i$. By equation~\eqref{eq:halfopendecompcones}, we have 
\[
\Ehr (P,w;t) = \Ehr (H_q S_1,w;t)+\cdots +\Ehr (H_q S_m,w;t).
\]
It thus suffices to prove the claimed rational form for all half-open simplices in the partition.

Let $\Delta=H_q S_i$ be a rational half-open simplex in the partition.  Let $v_1,\dots,v_{r+1} \in \Z^{d+1}$ be generators of the half-open simplical cone $\conep(\Delta)$. Since the triangulation of $P$ used only vertices of $P$, we can choose $v_1,\ldots, v_{r+1}\in \Z^{d+1}$ such that their last coordinates are all equal to $q$.

Since every degree-$m$ form is a sum of monomials, each of which is a product of linear forms, it furthermore suffices to consider the case when $w$ is a product of linear forms.  The weighted Ehrhart series is obtained by substituting $z_1=\cdots=z_d=1$ and $z_{d+1}=t$ into the generating function in Proposition~\ref{lem:formula}.  Thus 
\begin{equation*}
\Ehr(\Delta, w; t)  = \sum_{x \in \Pi \cap \Z^{d+1}} \left( t^{x_{d+1}} \sum_{\substack{
        I_1 \uplus \cdots \uplus I_{r+1} = [m]}} \prod_{i \in I_1}
    \ell_i(v_1)\cdots\prod_{i \in
I_{r+1}}\ell_i(v_{r+1})\prod_{j=1}^{r+1}\frac{A_{|I_j|}^{\lambda_j(x)}(t^q)}{(1-t^q)^{|I_j|+1}}\right).
    \end{equation*}
    where $\Pi$ is the half-open parallelepiped in $\conep(\Delta)$ and each $x \in \Pi$ is written $x = \lambda_1(x) v_1 + \cdots +
    \lambda_{r+1}(x) v_{r+1}$.

    Since $|I_1| + \cdots |I_{r+1}| + r+1 = m+r+1$, we have
    \begin{equation}
\label{eqn:commonDenom}
      \prod_{j=1}^{r+1} \frac{1}{(1-t^q)^{|I_j|+1}} = \frac{1}{(1-t^q)^{m+r+1}}.\end{equation}

    Then we have
    \begin{equation}
      \label{eqn:hstar}
h^*_{\Delta,w}(t)  = \sum_{x \in \Pi \cap \Z^{d+1}} t^{x_{d+1}} \sum_{\substack{
        I_1 \uplus \cdots \uplus I_{r+1} = [m]}} \prod_{i \in I_1}
    \ell_i(v_1)\cdots\prod_{i \in
      I_{r+1}}\ell_i(v_{r+1})\prod_{j=1}^{r+1}A_{|I_j|}^{\lambda_j(x)}(t^q).
     \end{equation}
     Thus the claim follows with $h^\ast _{P, w}(t)=h^\ast _{H_q S_1, w}(t)+\cdots + h^\ast _{H_q S_m, w}(t)$.
  \end{proof}

\begin{rem}
In the multivariate version of the weighted Ehrhart rational function, the denominators do not simplify nicely as in~\eqref{eqn:commonDenom}. When bringing all constituents of the multivariate generating function of $C(P)$ in a common denominator this affects the positivity of the numerator polynomial.
\end{rem}

\subsection{Nonnegativity}

We are now ready to prove the main theorem stated in the introduction.   Recall that $R_P$ is the semiring consisting of sums of
products of nonnegative linear forms on $P$ and $S_P$ is the semiring consisting of sums of nonnegative products of linear forms on $P$. 
\begin{thm}[Nonnegativity Theorem]
\label{thm:main}
  Let $P$ be a rational polytope.
  \begin{enumerate}
\item   If  the weight $w$ is a homogeneous element of $R_P$, then the
coefficients of $h_{P,w}^*(t)$ are nonnegative.
\item  If  the weight $w$ is a homogeneous element of $S_P$, then
  $h_{P,w}^*(t) \geq 0$ for $t \geq 0$.
  \end{enumerate}
  \end{thm}

\begin{proof}
Let $P$ be a rational polytope of dimension $r$.

For (1), it suffices to prove the statement when the weight is a product 
of nonnegative linear forms on $\conep(P)$. The proof follows from the argument given in the proof of Proposition~\ref{prop:rational} where $h^\ast _{P,w}(t)$ is expressed as a sum of polynomials $h^\ast _{\Delta ,w}(t)$ as given in Equation~\ref{eqn:hstar} where $\Delta$ ranges over all half-open simplices in a half-open triangulation of $P$. Each of the vectors $v_i$ in Equation~\ref{eqn:hstar} is a generator of $\conep(P)$. Thus, if $w\in R_P$, $h^\ast _{\Delta, w}(t)$ has nonnegative coefficients and so does $h^\ast _{P, w}(t)$ as a sum of these polynomials.

For (2), let $w$ be a product of linear forms
$\ell_1,\dots,\ell_m$ on $\conep(P)$, and assume $w$ is nonnegative on $P$.
First suppose $\ell_1,\dots,\ell_m$ all have rational coefficients.
Subdivide $P$ into rational polytopes using the hyperplanes $\ell_1=0, \dots, \ell_m=0$.  Let $s$ be a positive integer such that $sQ$ has integer coordinates for every $r$-dimensional polytope $Q$ that is part of the subdivision.  Then $s$ is divisible by the denominator $q=\delta (P)$ of $P$.  On each such polytope $Q$, each linear form $\ell_i$ is either entirely nonnegative or entirely nonpositive, and the number of nonpositive ones is even because their product $w$ is nonnegative.
Thus after changing the signs of an even number of the linear forms on $Q$, which does not change $w$, we can apply the part (1) result to obtain that
\[
\Ehr(Q,w;t) = \frac{h_Q(t)}{(1-t^s)^{r+m+1}} = \frac{h^\ast_{Q,w}(t)(1+t^{\delta (Q)}+\ldots +t^{s-\delta (Q)})^{r+m+1}}{(1-t^{\delta (Q)})^{r+m+1}(1+t^{\delta (Q)}+\ldots +t^{s-\delta (Q)})^{r+m+1}}
\]
where $h_Q(t)$ has nonnegative coefficients for every polytope $Q$ in the subdivision since $h^\ast_{Q,w}(t)$ has nonnegative coefficients by part (1).  The weight $w$ is zero on the boundaries where the polytopes overlap in the subdivision, so the Ehrhart series of $P$ is the sum of Ehrhart series of the $r$-dimensional polytopes in the subdivision.
Summing them up gives
\[
\Ehr(P,w;t) = \frac{h(t)}{(1-t^s)^{r+m+1}},
\]
for some polynomial $h(t)$ with nonnegative coefficients.
Since $s$ is divisible by the denominator $q$ of $P$, we have
\[
\frac{h^*_{P,w}(t)}{(1-t^q)^{r+m+1}} =\frac{h(t)}{(1-t^s)^{r+m+1}} =  \frac{h(t)}{((1-t^q)(1+t^q+t^{2q}+\dots+t^{s-q}))^{r+m+1}},
\]
so \[h^*_{P,w}(t)(1+t^q+t^{2q}+\cdots+t^{s-q})^{r+m+1} = h(t).\]  The polynomial $h(t)$ has nonnegative coefficients, so $h(t) > 0$ for $t > 0$.  It follows that $h_{P,w}^*(t) > 0$ for all $t > 0$.  This proves part (2) when the linear forms have rational coefficients.

To deal with irrational coefficients, note that for a fixed polytope $P$, the map that sends a weight $w$ to the corresponding $h^*$-polynomial $h^\ast _{P,w}(t)$ is a linear, hence continuous, map from the vector space of homogeneous degree $m$ polynomials to the vector space of degree $\leq r+m$ polynomials.  The set of polynomials $h^*$ satisfying $h^*(t) \geq 0$ when $t \geq 0$ is a closed set.  Thus we obtain the result (2) for linear forms with irrational coefficients as well.
\end{proof}

\subsection{Monotonicity}

In this subsection we generalize Stanley's monotonicity result for the $h^\ast$-polynomial for rational polytopes to a weighted version by proving Theorem~\ref{thm:monotone1}. Our proof follows a similar structure as the proof of nonnegativity. We start by proving a version of the claim for pyramids over half-open simplices and then extend it to all rational polytopes. This will become useful when comparing $h^\ast$-polynomials of polytopes of different dimension in the general case. 

Given a half-open $r$-dimensional rational simplex $F\subseteq \R^d$, say
\[
F=\left\{\lambda_1v_1+\cdots+\lambda_{r+1}v_{r+1} \mid  \lambda_1,\hdots,\lambda_k\geq 0, \lambda_{k+1},\hdots,\lambda_{r+1}>0, \lambda_1+\cdots+\lambda_{r+1}=1 \right\},
\]
and a rational point $u\in \R^d$ not in the affine span of $F$, we let the \emph{pyramid of $u$ over $F$} be
\begin{align*}
\Pyr(u,F):=\{ &\mu u+\lambda_1v_1+\cdots+\lambda_{r+1}v_{r+1}\mid
\\
&\mu,\lambda_1,\hdots,\lambda_k\geq 0, \lambda_{k+1},\hdots,\lambda_{r+1}>0, \mu+\lambda_1+\cdots+\lambda_{r+1}=1 \}.
\end{align*}
We denote the $s$-fold pyramid of $u_1,\hdots,u_s\in\Q^d$ over $F$ by 
\[
\Pyr^{(s)}(u_1,\hdots,u_s,F):=\Pyr(u_1,\Pyr(u_2,\hdots \Pyr(u_s,F))),
\]
now a half-open simplex of dimension $s+r$.

\begin{lem}
\label{lem:pyramid}
Let $F\subseteq \R^d$ be a half-open $r$-dimensional rational simplex with denominator $\delta(F)$ and let $\Delta$ be an $s$-fold pyramid over $F$ with denominator $\delta(\Delta)$. For all $g\geq1$ divisible by $\delta(\Delta)$ and all $w=\ell_1\cdots\ell_m \in R_\Delta$,
\[
(1+t^{\delta(F)}+\cdots+t^{g-\delta(F)})^{r+m+1}h^\ast_{F,w}(t)\preceq
(1+t^{\delta(\Delta)}+\cdots+t^{g-\delta(\Delta)})^{s+r+m+1}h^\ast_{\Delta,w}(t).
\]
\end{lem}
\begin{proof}
Let $v_1,\hdots,v_{r+1}\in\frac{1}{\delta(F)}\Z^d$ be vertices of $F$, labeled such that
\[
F=\left\{\lambda_1v_1+\cdots+\lambda_{r+1}v_{r+1} \mid \lambda_1,\hdots,\lambda_k\geq 0, \lambda_{k+1},\hdots,\lambda_{r+1}>0, \lambda_1+\cdots+\lambda_{r+1}=1 \right\}.
\]
Suppose $u_1,\hdots,u_s\in\frac{1}{\delta(\Delta)}\Z^d$ are such that  $\Delta=\Pyr^{(s)}(u_1,\hdots,u_s,F)$, that is, suppose
\begin{align*}
\Delta=\{ &\mu_1u_1+\cdots+\mu_su_s+\lambda_1v_1+\cdots+\lambda_{r+1}v_{r+1} \mid\\
&\mu_1,\hdots,\mu_s,\lambda_1,\hdots,\lambda_k\geq 0, \lambda_{k+1},\hdots,\lambda_{r+1}>0, \mu_1+\cdots+\mu_s+\lambda_1+\cdots+\lambda_{r+1}=1 \}.
\end{align*}

Considering the cone $\conep(F)$ with generators of last coordinate $g$ and fundamental parallelepiped
\[
\Pi_g(F)=\left\{ \lambda_1\begin{pmatrix} gv_1\\g\end{pmatrix}+\cdots+ \lambda_{r+1}\begin{pmatrix} gv_{r+1}\\g\end{pmatrix} \mid 0\leq \lambda_1,\hdots,\lambda_k<1, 0<\lambda_{k+1},\hdots,\lambda_{r+1}\leq 1\right\},
\]
we obtain by Proposition~\ref{lem:formula}
\begin{equation}
      \label{eqn:Fseries}
\Ehr(F,w;t)=\frac{\displaystyle \sum_{x\in \Pi_g(F)\cap\Z^{d+1}} t^{x_{d+1}} \sum_{\substack{
        I_1 \uplus \cdots \uplus I_{r+1} = [m]}} \prod_{i \in I_1}
    \ell_i(gv_1)\cdots\prod_{i \in
      I_{r+1}}\ell_i(gv_{r+1})\prod_{j=1}^{r+1}A_{|I_j|}^{\lambda_j(x)}(t^g)}{(1-t^g)^{r+m+1}}.
\end{equation}

Analogously, considering the cone $\conep(\Delta)$ with generators of last coordinate $g$ and fundamental parallelepiped
\begin{align*}
\Pi_g(\Delta)=\bigg\{ & \mu_1\begin{pmatrix} gu_1\\g\end{pmatrix}+\cdots+ \mu_{s}\begin{pmatrix} gu_s\\g\end{pmatrix} + \lambda_1\begin{pmatrix} gv_1\\g\end{pmatrix}+\cdots+ \lambda_{r+1}\begin{pmatrix} gv_{r+1}\\g\end{pmatrix} \mid \\
& 0\leq \mu_1,\hdots,\mu_s,\lambda_1,\hdots,\lambda_k<1, 0<\lambda_{k+1},\hdots,\lambda_{r+1}\leq 1 \bigg\},
\end{align*}
we obtain by Proposition~\ref{lem:formula} 
\begin{equation}
      \label{eqn:Deltaseries}
    \Ehr(\Delta,w;t)=
        \frac{
            \resizebox{.82\hsize}{!}{$
                \displaystyle
                   \sum_{x\in \Pi_g(\Delta)\cap\Z^{d+1}} t^{x_{d+1}} \sum_{\substack{I_1 \uplus \cdots \uplus I_{s+r+1} = [m]}} 
                    \prod_{i \in I_1} 
                    \ell_i(gv_1)\cdots\prod_{i \in I_{r+1}}\ell_i(gv_{r+1})
                    \prod_{i \in I_{r+2}}\ell_i(gu_1)\cdots
                    \prod_{i \in   I_{s+r+1}}\ell_i(gu_{s})
                    \prod_{j=1}^{s+r+1}A_{|I_j|}^{\lambda_j(x)}(t^g)$}}
                    {(1-t^g)^{s+r+m+1}}.
\end{equation}

Observe that $\Pi_g(F)\subseteq \Pi_g(\Delta)$. In particular, the points in $\Pi_g(F)$ are those in $\Pi_g(\Delta)$ with $\mu_1=\cdots=\mu_s=0$. Therefore, for every $x\in \Pi_g(F)\cap\Z^{d+1}$, each term in the inner sum of the numerator of~\eqref{eqn:Fseries} appears as a term of the numerator of~\eqref{eqn:Deltaseries} with $I_{r+2}=\cdots=I_{s+r+1}=\varnothing$ (where $\lambda_{r+1}(x)=\cdots=\lambda_{s+r+1}(x)=0$). Thus, since $w\in R_\Delta$, the nonnegativity of the remaining terms implies that
\[
(1-t^g)^{r+m+1}\Ehr(F,w;t)\preceq (1-t^g)^{s+r+m+1}\Ehr(\Delta,w;t).
\]
Recalling that the denominators of the Ehrhart series $\Ehr(F,w;t)$ and $\Ehr(\Delta,w;t)$ are $(1-t^{\delta(F)})^{r+m+1}$ and $(1-t^{\delta(\Delta)})^{s+r+m+1}$, respectively, we cancel these denominators and get the desired claim
\[
(1+t^{\delta(F)}+\cdots+t^{g-\delta(F)})^{r+m+1}h^\ast_{F,w}(t)\preceq
(1+t^{\delta(\Delta)}+\cdots+t^{g-\delta(\Delta)})^{s+r+m+1}h^\ast_{\Delta,w}(t).
\]
\end{proof}

We are now ready to prove the monotonicity theorems stated in the introduction.   Recall that $R_Q$ is the semiring consisting of sums of
products of nonnegative linear forms on $Q$. 

\begin{thm}[First Monotonicity Theorem]
    \label{thm:monotone1}
 Let $P,Q\subseteq\R^d$ be rational polytopes, $P\subseteq Q$, and let $g$ be a common multiple of the denominators $\delta (P)$ of $P$ and $\delta (Q)$ of $Q$. Then, for all weights $w\in R_Q$,
    \[
(1+t^{\delta(P)}+\cdots +t^{g-\delta (P)})^{\dim P+m+1}h^\ast _{P,w}(t) \preceq (1+t^{\delta(Q)}+\cdots +t^{g-\delta (Q)})^{\dim Q+m+1}h^\ast _{Q,w}(t) \, .
\]
    In particular, if $P \subseteq Q$ are polytopes with the same denominator, then taking $g=\delta(P)=\delta(Q)$ gives
    \begin{equation}
h^\ast _{P,w}(t) \preceq h^\ast _{Q,w}(t)
    \end{equation}
\end{thm}

\begin{proof}
If $P$ is empty, then $h^*_{P,w}(t) = 0$, so the statement becomes part (1) of the Nonnegativity Theorem (Theorem~\ref{thm:main}) above.  Now let us assume that $P$ is nonempty.
   We can extend a half-open triangulation of $P$ to a half-open triangulation of $Q$ as follows.

   Let $T$ be a half-open triangulation of $P$ into simplices of dimension $\dim{P}$ with denominators dividing $\delta(P)$. 
   Choose $u_1,\hdots,u_{s}\in Q\cap\frac{1}{g}\Z^d$, where $s=\dim{Q}-\dim{P}$, so that for each $F\in T$ the $s$-fold pyramid $\Delta_F=\Pyr^{(s)}(u_1,\hdots,u_{s},F)\subseteq Q$ is a half-open simplex of dimension $\dim{Q}$. This is always possible by, for example, starting with a triangulation of $P$ using no new vertices and choosing $u_1,\dots,u_s$ successively from the vertices of $Q$ that do not lie on the affine hull of the previous ones together with $P$. Let $\Pyr^{(s)}(P)$ denote the union of the $\Delta_F$ which form a half-open triangulation. By Lemma~\ref{lem:pyramid}, for every $F\in T$,
    \begin{equation}
    \label{eqn:pyrineq}
    (1+t^{\delta(F)}+\cdots+t^{g-\delta(F)})^{\dim{P}+m+1}h^\ast_{F,w}(t)\preceq
    (1+t^{\delta(\Delta_F)}+\cdots+t^{g-\delta(\Delta_F)})^{\dim{Q}+m+1}h^\ast_{\Delta_F,w}(t).
    \end{equation}
    The left-hand side of~\eqref{eqn:pyrineq} is equal to $(1-t^g)^{\dim{P}+m+1}\Ehr(F,w;t)$ and the right-hand side of~\eqref{eqn:pyrineq} is equal to $(1-t^g)^{\dim{Q}+m+1}\Ehr(\Delta_F,w;t)$. Therefore, summing over all $F\in T$ yields
    \begin{equation}
    \label{eqn:P,Pyr}
            (1-t^g)^{\dim{P}+m+1}\Ehr(P,w;t)\preceq (1-t^g)^{\dim{Q}+m+1}\Ehr(\Pyr^{(s)}(P),w;t).
        \end{equation}
    Next we extend the half-open triangulation of $\Pyr^{(s)}(P)$ to a half-open triangulation $T'$ of the entire polytope $Q$. This can be done by using a sequence of pushings (or placings) of the vertices of $Q$ that are not in $P$ to extend the triangulation of $\Pyr^{(s)}(P)$ to $Q$; see page 96 and Section 4.3 of \cite{triangbook} for more details. Using a generic point in Theorem~\ref{thm:halfopen} to be in the interior of $\Pyr^{(s)}(P)$ the resulting triangulation of $Q$ becomes half-open. Each half-open simplex in $T'$ has dimension $\dim{Q}$ and denominator dividing $g$. By Proposition~\ref{prop:rational}, for each $\Delta\in T'$, $(1-t^g)^{\dim{Q}+m+1}\Ehr(\Delta,w;t)$ is a polynomial with nonnegative coefficients. Therefore,
        \begin{equation}
    \label{eqn:Pyr,Q}
    (1-t^g)^{\dim{Q}+m+1}\Ehr(\Pyr^{(s)}(P),w;t) \preceq 
    (1-t^g)^{\dim{Q}+m+1}\Ehr(Q,w;t).
    \end{equation}
From \eqref{eqn:P,Pyr} and \eqref{eqn:Pyr,Q} it follows that
    \[
    (1-t^g)^{\dim{P}+m+1}\Ehr(P,w;t)\preceq
    (1-t^g)^{\dim{Q}+m+1}\Ehr(Q,w;t).
    \]
    Equivalently,
    \[
    (1+t^{\delta(P)}+\cdots+t^{g-\delta(P)})^{\dim{P}+m+1}h^\ast_{P,w}(t)\preceq (1+t^{\delta(Q)}+\cdots+t^{g-\delta(Q)})^{\dim{Q}+m+1}h^\ast_{Q,w}(t).
    \]
\end{proof}

 \begin{thm}[Second Monotonicity Theorem]
\label{thm:monotone2}
 Let $P,Q\subseteq\R^d$ be rational polytopes of the same dimension $D=\dim P=\dim Q$, $P\subseteq Q$, and let $g$ be a common multiple of the denominators $\delta (P)$ of $P$ and $\delta (Q)$ of $Q$. Then, for all weights $w\in S_Q$,
    \[
(1+t^{\delta(P)}+\cdots +t^{g-\delta (P)})^{D+m+1}h^\ast _{P,w}(t) \leq (1+t^{\delta(Q)}+\cdots +t^{g-\delta (Q)})^{D+m+1}h^\ast _{Q,w}(t) \, 
\]
for all $t\geq 0$.
    In particular, if $P \subseteq Q$ are polytopes with the same denominator and dimension, then taking $g=\delta(P)=\delta(Q)$ gives
    \begin{equation}
        \label{eqn:monotone1}
h^\ast _{P,w}(t) \leq h^\ast _{Q,w}(t)   \text{ for all }t\geq 0.
    \end{equation}
\end{thm}

\begin{proof}
Let $w$ be a product of linear forms $\ell_1,...,\ell_m$ on the homogenization $\conep(P)$ such that $w$ is nonnegative on $P$ and $\ell_1,\dots,\ell_m$ have rational coefficients. Now, let us use the hyperplanes $\ell_1 = 0,...,\ell_m = 0$, as in the proof of Theorem \ref{thm:main} (2), to subdivide $P$ and $Q$ into rational polytopes $P_1',\ldots, P_k'$ and $Q_1',\ldots, Q_k'$, $P_i'\subseteq Q_i'$, respectively. Note, if any of these polytopes in the subdivision has dimension smaller than $D$ then it is included in one of the hyperplanes and thus its $h^\ast$-polynomial is zero. Thus, we can compute the Ehrhart series of $P$ and $Q$ by summing up the series of those subpolytopes $P_i'$s and $Q_i's$ where $\dim(P_i') = \dim(Q') = D$, and we may assume that each $P_i'$ in the subdivision of $P$ that we consider is contained in a unique polytope $Q_i'$ in the subdivision of $Q$.   
    
    As before, every linear form $\ell_i$ with $1 \leq i \leq m$ is either entirely nonpositive or entirely nonnegative on each such polytope $P_i' \subseteq Q_i'$. Hence, we can change the signs of an even number of linear forms on $P_i'$ and $Q_i'$ without changing the weight $w$ since the product of these linear forms is nonnegative. 
    
    Let $g'$ be a positive integer multiple of all the denominators of $P_i'$s and $Q_i's$ in the subdivisions that additionally is also a multiple of $g$.
    We may now apply Theorem~\ref{thm:monotone1} to all $P_i'\subseteq Q_i'$ and obtain that
    \[
    (1+t^{\delta(P_i')}+\cdots +t^{g'-\delta (P_i')})^{D+m+1}h^\ast _{P_i',w}(t) \preceq (1+t^{\delta(Q_i')}+\cdots +t^{g'-\delta (Q_i')})^{D+m+1}h^\ast _{Q_i',w}(t) \, .
    \] 
    We can rewrite this as

    \[
    (1-t^{g'})^{D+m+1} \Ehr(P_i',w;t) \preceq (1-t^{g'})^{D+m+1} \Ehr(Q_i',w;t) \, .
    \]

    Since the weight $w$ is zero on the boundaries of the subdivision given by the linear forms $\ell _1,\ldots, \ell _m$, we can add up the inequalities for all pairs of polytopes $P_i\subseteq Q_i$ obtaining the following
    \begin{equation}\label{ineq:monotonicity2}
        (1 - t^{g'})^{D+m+1} \Ehr(P,w;t) \preceq (1-t^{g'})^{D+m+1} \Ehr(Q,w;t) \, .
    \end{equation}
    The left hand side of the inequality~\eqref{ineq:monotonicity2} equals
    \[
(1+t^g+\cdots +t^{g'-g})^{D+m+1}(1+t^{\delta(P)}+\cdots +t^{g-\delta (P)})^{D+m+1}h^\ast _{P,w}(t)
    \]
    and similarly for $Q$. Thus, we obtain that the polynomial $(1+t^g+\cdots +t^{g'-g})^{D+m+1}$ multiplied with
        \begin{equation}
(1+t^{\delta(Q)}+\cdots +t^{g-\delta (Q)})^{D+m+1}h^\ast _{Q,w}(t) -(1+t^{\delta(P)}+\cdots +t^{g-\delta (P)})^{D+m+1}h^\ast _{P,w}(t)
    \end{equation}
    has only nonnegative coefficients. In particular, evaluations at $t\geq 0$ of the product are nonnegative. Since $(1+t^g+\cdots +t^{g'-g})^{D+m+1}>0$ the nonnegativity of the evaluation of the second factor at nonnegative reals follows.

For linear forms with irrational coefficients as well as for an arbitrary element of $S_P$, we can argue again by linearity and continuity of the coefficients of the $h^\ast$-polynomials as in the proof of Theorem~\ref{thm:monotone1}. 
\end{proof}

Unlike the unweighted case of Stanley~\cite{StanleyMonotonicity}
    the following example shows that the monotonicity in \eqref{eqn:monotone1} 
    need not hold when the polytopes do not have the same dimension:
\begin{exam}\label{ex:monotonicity2}   Consider $w=\ell ^2$ for $\ell (x)=2x_1+3x_2$, $v_1=(3,-2),v_2=(2,-2),v_3=(2,-1)$, $P=\conv (v_1,v_2)$, $Q=\conv (v_1,v_2,v_3)$.
    We have $\ell (v_1)=0, \ell (v_2)=-2, \ell (v_3)=1$. Both $P$ and $Q$ are unimodular simplices, thus there is only one lattice point in the fundamental parallelepiped, namely $0$. Thus, by Lemma~\ref{lem:pointcontribution} with all $\lambda _i=0$, we obtain
\[
h^\ast _{Q,w}(t)=t^2(\ell(v_1)+\ell(v_2)+\ell(v_3))^2 +t (\ell (v_1)^2+\ell (v_2)^2+\ell (v_3)^2)=t^2+5t
\]
\[
h^\ast _{P,w}(t)=t^2(\ell(v_1)+\ell(v_2))^2 +t (\ell (v_1)^2+\ell (v_2)^2)=4t^2+4t
\]
Thus, the coefficients of the $h^*$ polynomials are not monotone, and neither are the values since $h^\ast _{Q,w}(1)=6<8=h^\ast _{P,w}(1)$.
\qed
\end{exam}

\begin{rem} As was shown in Example~\ref{ex:monotonicity2}, the monotonicity in \eqref{eqn:monotone1} does not need to hold for rational polytopes $P,Q\subseteq\R^d$, $P\subseteq Q$, of different dimension. In this case, the same arguments as in the proof of Theorem~\ref{thm:monotone2} nevertheless yield the existence of an integer 
$g$ divisible by $\delta(P)$ and $\delta(Q)$ such that
\[
(1+t^{\delta(P)}+\cdots +t^{g-\delta (P)})^{\dim P+m+1}h^\ast _{P,w}(t) \leq (1+t^{\delta(Q)}+\cdots +t^{g-\delta (Q)})^{\dim Q+m+1}h^\ast _{Q,w}(t)
\]
for all $t \geq 0$ if the linear forms involved in the weight function have rational coefficients. Here we are no longer able to choose any $g$ divisible by $\delta(P)$ and $\delta(Q)$, as
the integer $g$ depends on linear forms involved.
\end{rem}

\section{Squares of arbitrary linear forms}\label{sec:2d}

In this section we focus on weights given as squares of arbitrary linear forms, not necessarily in $R_P$ and $h^\ast$-polynomials of polygons in the plane, and strengthen Theorem~\ref{thm:main} in this special case. We prove that if $P$ is a convex lattice polygon and the weight $w(x)=\ell(x)^2$ is given by a square of a linear form $\ell (x)$ then the coefficients of $h^\ast _{P,w}(t)$ are nonnegative, regardless of whether $\ell (x)$ is nonnegative on $P$ or not. This result is established in Theorem \ref{thm:polygon} below. This is a reformulation of results on the positivity of Ehrhart tensor polynomials of lattice polytopes considered in~\cite{Berg}. See Section~\ref{sec:tensors} below. Here, we present a proof that is arguably more elementary. We also present examples that show the limitations of our results if the conditions on the degree, dimension, denominator or convexity are removed.

\subsection{Lattice polygons}
We begin by providing the following more concise version of Equation (\ref{eqn:hstar}) in the case of the weight being given as a square of a linear form that holds in any dimension.

\begin{lem}\label{lem:pointcontribution}
Let $\ell: \R^d\rightarrow \R$ be a linear form. The $h^*$-polynomial $h^\ast_{\Delta,w}(t)$ with respect to the weight $w = \ell^2$ of any rational simplex $\Delta = \conv\{u_0,\ldots,u_r\}$ with denominator $q$ is given by the sum of the contributions
\begin{equation}\label{eq:contribution}
	\resizebox{.9 \textwidth}{!}{$q^2\left(\left(\sum(1-\lambda_i)\ell(u_i)\right)^2t^{2q}+\left(\sum \ell^2(u_i)+\left(\sum\ell(u_i)\right)^2
			-\left(\sum\lambda_i\ell(u_i)\right)^2-\left(\sum(1-\lambda_i)\ell(u_i)\right)^2\right)t^q+\left(\sum\lambda_i\ell(u_i)\right)^2\right)t^{x_{d+1}}$}
\end{equation}
  of each lattice point $x = \sum \lambda_i(x)(qu_i,q) \in \Pi(\Delta)\cap \Z^{d+1}$ in the fundamental parallelepiped where all summations are taken for indices $i$ from $0$ to $r$. 
\end{lem}

\begin{proof}
If $w(x)=\ell (x)^2$ then the weight is a product of $m=2$ linear forms and the contributions of each lattice point in the fundamental parallelepiped given in Equation~\eqref{eqn:hstar} is a linear combination of products of $A_0^\lambda (t)=1$, 
	\[
		A_2^\lambda(t)=(1-\lambda)^2t^2+(1+2\lambda-2\lambda^2)t+\lambda^2 \qquad \text{and} \qquad A_1^\lambda(t)=(1-\lambda)t+\lambda
	\]
	for $0\leq \lambda \leq 1$. More precisely, we use the homogenized linear form $\ell'$ associated with $\ell$ that takes in account the scaling factor in Equation~\eqref{eqn:hstar}. Then $\ell'(qu_i,q)=q\ell'(u_i,1)=q\ell(u_i)$ and we get that the contribution of any such point $x=\sum\lambda_i (qu_i,q)$ is
	\[
		q^2\left(\sum_{0\leq i\leq r} A_2^{\lambda_i}(t^q)\ell^2(u_i)+2\sum_{0\leq i<j\leq r}A_1^{\lambda_i}(t^q)A_1^{\lambda_j}(t^q)\ell(u_i)\ell(u_j)\right) t^{x_{d+1}} \, ,
	\]
	where the first sum corresponds to the ordered partitions $[2]=I_0\uplus I_1\uplus \cdots \uplus I_r$ into $r+1$ parts where $|I_i|=2$ for some $i$ and the second sum corresponds to partitions for which $|I_i|=|I_j|=1$ for some $i\neq j$.
	
	The factor $q^2$ is present in both cases. The coefficients of $t^{2q}$ and $1$ (times $t^{x_{d+1}}$) of the polynomial above are easily seen. Indeed, the first sum contributes $\sum (1-\lambda_i)^2\ell^2(u_i)$ and the second sum contributes
	$2\sum(1-\lambda_i)(1-\lambda_j)\ell(u_i)\ell(u_i)$ to the coefficient of $t^{2q}$.  Combining these, we obtain $\left(\sum(1-\lambda_i)\ell(u_i)\right)^2$ as claimed. Analogous arguments yield the coefficient of $1$ of every contribution.
	
	A similar analysis gives that the coefficient of $t^q$ is equal to
	\[
		\begin{aligned}
			\sum_i (1+&2\lambda_i-2\lambda_i^2)\ell^2(u_i)+2\sum_{i<j}\Bigl((1-\lambda_i)\lambda_j+(1-\lambda_j)\lambda_i\Bigr)\ell(u_i)\ell(u_j)\\
					&=\sum_i (1+2\lambda_i-2\lambda_i^2)\ell^2(u_i)+2\left(\sum_i\lambda_i\ell(u_i)\right)\left(\sum_j(1-\lambda_j)\ell(u_j)\right)-2\sum_i\lambda_i(1-\lambda_1)\ell^2(u_i)\\
					&=\sum_i\ell^2(u_i)+2\left(\sum_i\lambda_i\ell(u_i)\right)\left(\sum_j(1-\lambda_j)\ell(u_j)\right).
		\end{aligned}
	\]
	By squaring both sides of the identity 
	\[
		\sum_i\ell(u_i)=\left(\sum_i\lambda_i\ell(u_i)\right)+\left(\sum_j(1-\lambda_j)\ell(u_j)\right)
	\]
	we get the claimed coefficient of $t^q$.
\end{proof}

 \begin{lem}\label{lem:halfopenunit}
 Let $\Delta\subseteq\R^2$ be a half-open triangle with vertices in $\Z^2$, let $\ell: \R^2\rightarrow \R$ be a linear form and let $w(x)=\ell ^2 (x)$. If the $h^*$-polynomial $h^\ast_{\Delta,w}(t)$ of $\Delta$ with respect to $w(x)=\ell^2(x)$ has negative coefficients then the following two conditions must both be satisfied.
 \begin{itemize}
     \item[(i)] $\Delta$ is neither completely closed nor completely open, and
     \item[(ii)] the line $\ker \ell$ intersects the relative interior of two sides of $\Delta$ that are either both ``open'' or both ``closed''.
 \end{itemize}
 \end{lem}
 \begin{proof}
 Let $u_0,u_1,u_2$ be the vertices of $\Delta$. We argue by induction over the area of $\Delta$.

 We begin by assuming that $\Delta$ has area $\sfrac{1}{2}$, the minimal area among all triangles with vertices in the integer lattice. In this case, the half-open fundamental parallelepiped $\Pi (\Delta)$ contains exactly one lattice point $x=\lambda _0 (u_0,1)+\lambda _1 (u_1,1)+\lambda _2 (u_2,1)$ where $\lambda _0,\lambda _1,\lambda _2 \in \{0,1\}$.

 If $\Delta$ is completely closed then $\lambda _0=\lambda _1=\lambda _2=0$ and by Lemma~\ref{lem:pointcontribution},
\[
	h^\ast_{\Delta,w}(t)=\left(\sum\ell(v_i)\right)^2t^{2}+\left(\sum \ell(v_i)^2\right)t \, .
\]
 Similarly, if $\Delta$ is completely open, then $\lambda _0,\lambda _1=\lambda _2=1$ and 
\[
	h^\ast_{\Delta,w}(t)=\left(\sum\ell(u_i)\right)^2t^{3}+\left(\sum \ell(u_i)^2\right)t^4
\]
In particular, in both cases we see that the $h^\ast$-polynomial has only nonnegative coefficients. Thus, if a half-open lattice triangle has a negative coefficient condition (i) needs to be satisfied, that is, $\Delta$ is neither completely open nor closed. In this case, $\lambda_0,\lambda _1,\lambda_2$ are not all equal. 

We consider the case $\lambda _0=\lambda _1=0$ and $\lambda _2=1$. Then, by Lemma~\ref{lem:pointcontribution}, 
\begin{multline*}
	h^\ast_{\Delta,w}(t)=\\=
	\resizebox{0.95\textwidth}{!}{$(\ell(u_0)+\ell(u_1))^2t^{3}+\Bigl(\ell^2(u_0)+\ell^2(u_1)+\ell^2(u_2)+(\ell(u_0)+\ell(u_1)+\ell(u_2))^2-\ell^2(u_2)-(\ell(u_0)+\ell(u_1))^2\Bigr)t^2+\ell^2(u_2)t.$}
\end{multline*}
    The first and last coefficient are squares and thus always nonnegative. The coefficient of $t^2$ can be simplified to
\[
    (\ell(u_0)+\ell(u_2))^2+(\ell(u_1)+\ell(u_2))^2-\ell^2(u_2) \, .
\]
    We observe that if $\ell (u_2)$ has the same sign as $\ell (u_i)$, $i=0,1$, then $(\ell(u_i)+\ell(u_2))^2-\ell^2(u_2)\geq 0$ and thus the coefficient is nonnegative. It follows that $h_{\Delta,w}(t)$ can have a negative coefficient only if $\ell (u_2)$ has a different sign than both $\ell (u_0)$ and $\ell (u_1)$, that is, $\ker \ell$ separates $u_2$ from $u_0$ and $u_1$ as claimed. The case $\lambda _0=\lambda _1=1$ and $\lambda _2=0$ follows analogously. This proves the claim if $\Delta$ has minimal area.

    Now we assume that $\Delta$ has area greater than $\sfrac12$ and that the result has already been proved for all $\Delta$ of smaller area. In order to prove the claim it suffices to show that if $\Delta$ does not satisfy at least one of the conditions (i) or (ii) then it can be partitioned into half-open triangles that have $h^\ast$-polynomials with only nonnegative coefficients; then, by additivity also the $h^\ast$-polynomial of $\Delta$ is nonnegative and the proof will follow.
    
    If $\Delta$ has area greater than $\sfrac12$ then it contains at least one lattice point aside of its vertices, either in the relative interior of a side or in the interior of the triangle. By coning over the sides in which this point is not contained we obtain a subdivision into two or three smaller lattice triangles. By induction hypothesis it suffices to show that this subdivision can be made half-open in such a way that the half-open triangles in the partition do not satisfy both condition (i) and (ii).

    This is indeed always possible. In Figure~\ref{fig:quad_linear} the case of an interior lattice point and a subdivision into three smaller triangles is considered. The first row shows how to partition a completely closed triangle into smaller triangles that violate conditions (i) or (ii), depending on the position of $\ker \ell$. If $\Delta$ is completely open, then open and closed sides are flipped. The second row shows how such a partition is established in case $\Delta$ is half-open but $\ker \ell$ intersects in an open and a closed side. The non-intersected side can be removed in the case that it is excluded. 

    The case of a partition into two triangles can be treated in a similar way.
 \end{proof}

   \begin{figure}
    \centering
    \includegraphics[width=0.7\textwidth]{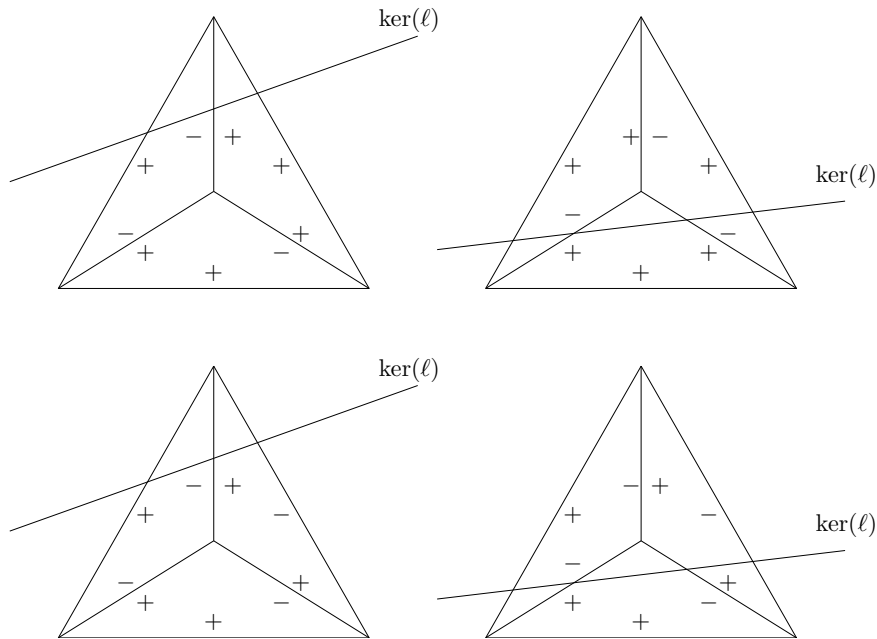}
    \caption{Subdivision of triangle using an internal integer point.
    		Each edge is marked with $+$ or $-$ to indicate which simplex includes it; the simplex containing $+$ contains the edge and the simplex containing $-$ excludes it.}
    \label{fig:quad_linear}
\end{figure}

\begin{thm}\label{thm:polygon}
For every (closed) convex lattice polygon $P$ and every linear form $\ell$, the $h^*$-polynomial of $P$ with respect to $w(x)=\ell^2(x)$ has only non-negative coefficients.
\end{thm}
\begin{proof}
If $\ker \ell$ does not intersect the interior of $P$, then the statement follows from Theorem~\ref{thm:main}. Otherwise, $\ker \ell$ intersects the boundary of $P$ twice: either in two vertices, or in a vertex and the interior of a side, or the interior of two sides. 

If $\ker \ell$ intersects the boundary of $P$ in two vertices, then the $h^\ast$-polynomial of $P$ is the sum of the $h^\ast$-polynomials of the two (closed) lattice polygons $\ker \ell$ divides $P$ into. This is because lattice points in $\ker \ell$ are weighted with $0$. The $h^\ast$-polynomial of both lattice polygons in the subdivision have only nonnegative coefficients by Theorem~\ref{thm:main} and so does their sum.

In the other two cases, if $\ker \ell$ intersects in a vertex and the interior of a side, or in the interior of two sides, the polygon can be subdivided into half-open triangles that do not satisfy the conditions (i) and (ii) in Lemma~\ref{lem:halfopenunit} as depicted in Figure~\ref{fig:polygonthm}: if the convex hull of the corresponding vertex and side/the two sides is a triangle, we take this closed triangle and extend it to a half-open triangulation as shown in the picture; if the convex hull of the two intersected sides i a quadrilateral, we partition this quadrilateral into a closed triangle and a half-open one along its diagonal; the rest of the polygon is again subdivided into half-open triangles that do not intersect $\ker \ell$, as depicted.

In all cases, the half-open triangles used in the half-open triangulation violate the conditions given in Lemma~\ref{lem:halfopenunit}. Thus their $h^\ast$-polynomial have only nonnegative coefficients and so does their sum. 
\end{proof}

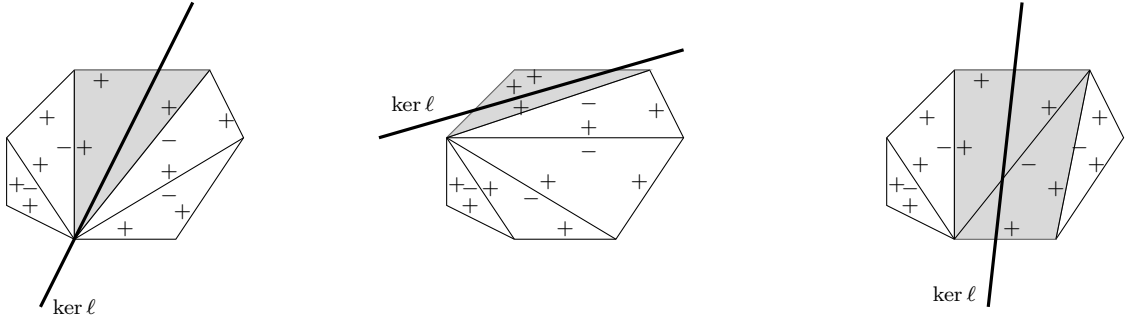
\begin{figure}
\begin{tikzpicture}[scale=0.45]
   \draw (2,0)--(5,0)--(7,3)--(6,5)--(2,5)--(0,3)--(0,1)--(2,0);
        \draw (15,0)--(18,0)--(20,3)--(19,5)--(15,5)--(13,3)--(13,1)--(15,0);
        \draw (28,0)--(31,0)--(33,3)--(32,5)--(28,5)--(26,3)--(26,1)--(28,0);

        \filldraw[draw={gray}, fill=gray!30] (2,0) -- (6,5) --(2,5)--(2,0);
        \filldraw[draw={gray}, fill=gray!30] (13,3) -- (15,5) --(19,5)--(13,3);
         \filldraw[draw={gray}, fill=gray!30] (28,0) -- (31,0) --(32,5)--(28,5)--(28,0);

\draw (2,0)--(0,3);
\draw (2,0)--(2,5);
\draw (2,0)--(6,5);
\draw (2,0)--(7,3);

\draw[very thick] (1,-2)--(5.5,7);

\draw (13,3)--(19,5);
\draw (13,3)--(20,3);
\draw (13,3)--(18,0);
\draw (13,3)--(15,0);

\draw[very thick] (11,3)--(20,5.6);

\draw (26,3)--(28,0);
\draw (28,5)--(28,0);
\draw (28,0)--(32,5);
\draw (31,0)--(32,5);

\draw[very thick] (29,-2)--(30,7);

\draw (3.5,0.3) node {\tiny $+$};
\draw (5.2,0.8) node {\tiny $+$};
\draw (4.8,1.3) node {\tiny $-$};
\draw (4.8,2) node {\tiny $+$};
\draw (6.5,3.5) node {\tiny $+$};
\draw (4.8,2.9) node {\tiny $-$};
\draw (4.8,3.9) node {\tiny $+$};
\draw (2.8,4.7) node {\tiny $+$};
\draw (2.3,2.7) node {\tiny $+$};
\draw (1.7,2.7) node {\tiny $-$};
\draw (1.2,3.6) node {\tiny $+$};
\draw (1,2.2) node {\tiny $+$};
\draw (0.7,1.5) node {\tiny $-$};
\draw (0.7,1) node {\tiny $+$};
\draw (0.3,1.6) node {\tiny $+$};

\draw (2,-2) node {\tiny $\ker \ell$};

\draw (13.7,1.5) node {\tiny $-$};
\draw (13.7,1) node {\tiny $+$};
\draw (13.3,1.6) node {\tiny $+$};
\draw (16.5,0.3) node {\tiny $+$};
\draw (14.3,1.5) node {\tiny $+$};
\draw (15.5,1.2) node {\tiny $-$};
\draw (16,1.7) node {\tiny $+$};
\draw (18.7,1.7) node {\tiny $+$};
\draw (17.2,2.6) node {\tiny $-$};
\draw (17.2,3.3) node {\tiny $+$};
\draw (19.2,3.8) node {\tiny $+$};
\draw (17.2,4) node {\tiny $-$};
\draw (15.2,3.9) node {\tiny $+$};
\draw (15.6,4.8) node {\tiny $+$};
\draw (15,4.5) node {\tiny $+$};

\draw (12,4) node {\tiny $\ker \ell$};

\draw (30.8,3.9) node {\tiny $+$};
\draw (28.8,4.7) node {\tiny $+$};
\draw (28.3,2.7) node {\tiny $+$};
\draw (27.7,2.7) node {\tiny $-$};
\draw (27.2,3.6) node {\tiny $+$};
\draw (27,2.2) node {\tiny $+$};
\draw (26.7,1.5) node {\tiny $-$};
\draw (26.7,1) node {\tiny $+$};
\draw (26.3,1.6) node {\tiny $+$};
\draw (29.7,0.3) node {\tiny $+$};
\draw (31,1.5) node {\tiny $+$};
\draw (30.2,2.2) node {\tiny $-$};
\draw (31.7,2.7) node {\tiny $-$};
\draw (32.2,2.2) node {\tiny $+$};
\draw (32.3,3.6) node {\tiny $+$};

\draw (28,-1.6) node {\tiny $\ker \ell$};

\end{tikzpicture}
    \caption{Half-open triangulations of a polygon in the cases where $\ker \ell$ intersects the boundary of the polygon in a vertex and the interior of a side (left) or two sides (middle/right). Removed/open faces are denoted by ``$-$'', closed/non-removed faces with ``+''. The convex hull of the corresponding vertex/sides is depicted in gray. All half-open triangles violate conditions (i) and (ii) of Lemma~\ref{lem:halfopenunit}.}
    \label{fig:polygonthm}
\end{figure}

\subsection{Negative examples}
\label{sec:negative}
In this section we provide examples that show that most assumptions in Theorem~\ref{thm:polygon} are necessary and cannot be further relaxed. Our examples are explicit and can be computed either by applying Equation (\ref{eqn:hstar}) and/or by using {\tt LattE} (\cite{LattE}). 

We begin with an example that shows that the nonnegativity of the $h^\ast$-polynomial for lattice polygons does not extend to weight functions that are squares of degree higher than $2$.

\begin{exam}\label{exam:2d-negative}
Let $w(x)=(2x_1-x_2)^2(2x_2-x_1)^2$ and $P$ be the standard triangle with vertices $v_0=(0,0), v_1= (1,0)$, and  $v_2 = (0,1)$. Then 
\[
	h^*_{P,w}(t)=t(8+81t-6t^2+t^3).
\]
\end{exam}

While the classical Ehrhart theory deals with convex polytopes, in the two-dimensional case, Stanley's nonnegativity theorem and our Theorem \ref{thm:main} can be extended to non-convex polygons without holes as any such polygon can be dissected into (half-open) triangles. Next we give an example of a non-convex quadrilateral and weight given by a square of a linear form that shows that Theorem~\ref{thm:polygon} does not extend to non-convex quadrilaterals.

\begin{exam}\label{exam:non-convex}
Let $w(x)=\ell (x)^2$ where $\ell(x)=x_1$ and $P=v_0v_1v_2v_3$ be the non-convex quadrilateral with vertices $v_0=(1,0),v_1=(-3,-1),v_2=(2,0),v_3=(-3,1)$ as depicted in Figure~\ref{fig:b}. Then
\[
	h^*_{P,w}(t)=t(23-4t+9t^2).
\]
\end{exam}

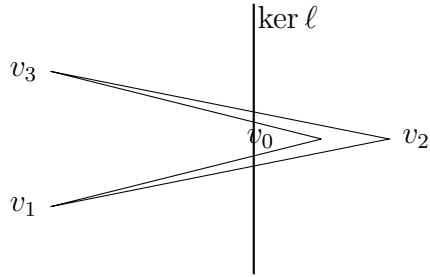
\begin{figure}[H]
        \begin{tikzpicture}[scale=0.9]
\draw (1,0) -- (-3,-1) --  (2,0) -- (-3,1) -- (1,0);
\draw[thick] (0,-2)--(0,2);
\node (v1) at (0.1,0) {$v_0$};
\node (v2) at (-3.4,-1) {$v_1$};
\node (v3) at (2.4,0) {$v_2$};
\node (v4) at (-3.4,1) {$v_3$};
\node (ell) at (0.5,1.8) {$\ker \ell$};
    \end{tikzpicture}
    \caption{Example of a non-convex lattice quadrilateral that has an $h^\ast$-polynomial with negative coefficients with respect to a weight $w(x)=x_1^2$.}
    \label{fig:b}
\end{figure}

Next, we note that Theorem \ref{thm:polygon} does not hold for rational polygons, not even in the case of ``primitive'' triangles as illustrated in the next example.

\begin{exam}
For any integer $q\geq 1$, let $\Delta _q \subseteq \mathbb{R}^2$ be the rational triangle with vertices 
\[
u_0=(1,1) \, , u_1=\left(1,\frac{q-1}{q}\right) \, \text{ and } u_2=\left(\frac{q+1}{q},1\right)
\]that has denominator $q$. Let $\ell _q\colon \mathbb{R}^2\rightarrow \mathbb{R}$ be the linear form defined by $\ell _q (x)=2q(1-q)x_1+q(2q-1)x_2$. Then
\begin{eqnarray*}
    \ell _q (u_0) &=& q \\
    \ell _q (u_1) &=& 1-q \\
    \ell _q (u_2) &=& 2-q \, .
\end{eqnarray*}

The half-open fundamental parallelepiped spanned by $(qu_0,q),(qu_1,q),(qu_2,q)$ contains exactly $q$ lattice points, namely
\[
y_i \ = \ (i,i,i) \quad \text{ for all }0\leq i\leq q-1 \, .
\]

By Lemma~\ref{lem:pointcontribution} we see that every non-zero coefficient of the $h^\ast$-polynomial of $\Delta$ with respect to $w_q(x)=\ell _q (x)^2$ arises from the contribution of exactly one of the $y_i$s, namely $y_i$ contributes to the coefficient of $t^j$ if and only if $j \equiv i \mod q$. Thus, $h^\ast _{\Delta _q,w_q}(t)$ has a negative coefficient if and only if the contribution of one of the lattice points in the half-open parallelepiped has a negative coefficient.

We focus on 
\[
y_{q-1}=(q-1,q-1,q-1)=\frac{q-1}{q}(qu_0,q)+0\cdot (qu_1,q)+0\cdot (qu_2,q).
\]
By Lemma~\ref{lem:pointcontribution}, the second term in the contribution of $y_{q-1}$, and therefore the coefficient of $t^{2q-1}$, is equal to $q^2$ times
\begin{eqnarray*}
    &q^2+(1-q)^2+(2-q)^2+(q+(1-q)+(2-q))^2-\left( \frac{q-1}{q}q\right)^2-\left(\frac{1}{q}q+(1-q)+(2-q) \right)^2\\
\end{eqnarray*}
which is equal to $-q^4+6q^3-3q^2$. This evaluates to a negative number for all integers $q\geq 6$. As a consequence, the $h^\ast$-polynomial of $\Delta _q$ with respect to the weight $w_q(x)=\ell _q (x)^2$ has a negative coefficient in front of $t^{2q-1}$ for all integers $q\geq 6$. For example, if $q=6$ then $h^\ast _{\Delta _q,w_q}(t)$ equals
\begin{eqnarray*}
    &&2304 t^{17} + 1764 t^{16} + 1296 t^{15} + 900 t^{14} + 576 t^{13} + 324 t^{12}- 108 t^{11} + 756 t^{10}\\ && + 1476 t^9 + 2052 t^8 + 2484 t^7 + 2772 t^6 + 900 t^5 + 576 t^4 + 324 t^3 + 144 t^2 + 36 t
\end{eqnarray*}
\end{exam}

Last but not least, we show that the assumption on the dimension cannot be removed in Theorem~\ref{thm:polygon} by providing an example of a $20$-dimensional lattice simplex $P$ and a linear form such that $h^\ast_{P,w}(x)$ has a negative coefficient where $w(x)=\ell(x)^2$. This also establishes a counterexample to a conjecture of Berg, Jochemko, Silverstein~\cite{Berg}, see Section~\ref{sec:tensors} below for details.
\begin{exam}
\label{ex:negative}
    We consider the $19$-dimensional simplex $\Delta=\conv\{u_0,\ldots,u_{19}\}$ where $u_0$ is the origin, $u_1,\ldots,u_{18}$ are the standard basis vectors $e_1,\ldots,e_{18}$ and 
    \[
	\begin{aligned}
		u_{19} &= (1,1,1,1,1,1,1,1,1,-1,-1,-1,-1,-1,-1,-1,-1,-1,3) \\
				&= 3e_{19}+e_1+\ldots+e_9-e_{10}-\ldots-e_{18}.
	\end{aligned}
    \]
    and the pyramid $\Delta ' = \conv \left(0\cup \Delta \times 1\right)\in \mathbb{R}^{20}$ which is a $20$-dimensional simplex with vertices $0$ and $v_i:=(u_i,1)$, $0\leq i\leq 19$. Let $\ell : \R^{20}\rightarrow \R$ be the linear functional defined by
    \[
\ell (v_i) = \begin{cases} 1 & \text{ if } 0\leq i\leq 9\\
-1 & \text{ if } 10\leq i \leq 19 \, .

\end{cases}
    \]
    We claim that the $h^\ast$-polynomial of $\Delta '$ with respect to $w(x)=\ell (x)^2$ has a negative coefficient in front of~$t^{11}$.

    To see this, we observe that the determinant of the matrix with columns $v_i$, $0\leq i\leq 19$ equals $-3$, that is, the normalized volume of $\Delta '$ is $3$ and the half-open fundamental parallelepiped $\Pi (\Delta ')$ contains exactly three lattice points. Those are $y_0=0$,
    \begin{eqnarray*}
y_1 \ &=& \ \frac{2}{3}\sum _{i=0}^9 (v_i,1)+\frac{1}{3}\sum _{i=10}^{19} (v_i,1) \ = \ (1,1,1,1,1,1,1,1,1,0,0,0,0,0,0,0,0,0,1,10,10) \, ,\\
y_2 \ &=& \ \frac{1}{3}\sum _{i=0}^9 (v_i,1)+\frac{2}{3}\sum _{i=10}^{19} (v_i,1) \ = \ (1,1,1,1,1,1,1,1,1,0,0,0,0,0,0,0,0,0,2,10,10) \, .
    \end{eqnarray*}

    By Lemma~\ref{lem:pointcontribution}, the coefficient of $t^{11}$ in the contribution of $y_j$, $j=1,2$, equals
    \[
\sum \ell^2(v_i)+(\sum \ell (v_i))^2-(\sum \lambda _i \ell (v_i))^2-(\sum (1-\lambda _i)\ell (v_i))^2
    \]
    where $\lambda _0=\cdots =\lambda _9=2/3$ and $\lambda _{10}=\cdots =\lambda _{19}=1/3$ for $y_1$, and for $y_2$ the values are flipped. In both cases, the term evaluates to 
    \[
    20+(0)^2-\left(\frac{2}{3}\cdot 10+\frac{1}{3}\cdot (-10)\right)^2-\left(\frac{1}{3}\cdot 10+\frac{2}{3}\cdot (-10)\right)^2 = \frac{-20}{9} \, .
    \]
    Note that $y_0=0$ does not contribute to the $t^{11}$-coefficient of the $h^\ast$-polynomial. In summary, the coefficient of $t^{11}$ equals $2\cdot \frac{-20}{9}<0$ and is thus negative.
\end{exam}

\section{Ehrhart tensor polynomials}\label{sec:tensors}
In this section we discuss the results of the previous section in relation to results and conjecture on Ehrhart tensor polynomials which were introduced by Ludwig and Silverstein~\cite{Ludwig}. 

For any integer $r\in \N$, let $\Te^r$ be the vector space of symmetric tensors of rank $r$ on $\R^d$. The \emph{discrete moment tensor} of rank $r$ of a  lattice polytope $P\subset \R^d$ is defined as
\[
L^r(P) \ = \ \sum _{x\in P\cap \Z^d} x^{\otimes r} \, ,
\]
where $x^{\otimes r}=x\otimes \cdots \otimes x$ and $x^{\otimes 0}:=1$. Discrete moment tensors were introduced by B\"or\"oczky and Ludwig~\cite{BL17}. Note that for $r=0$ we recover the number of lattice points in $P$, $|P\cap \Z^d|$. Ludwig and Silverstein~\cite[Theorem 1]{Ludwig} showed that there exist maps $L_i^r$, $0\leq i\leq d+1$, from the family of lattice polytopes to $\Te^r$ such that
\[
L^r(nP)=\sum _{i=0}^{d+r}L_i^r(P)n^i
\]
for all integers $n\geq 0$, that is, the discrete moment tensor $L^r(nP)$ is given by a polynomial in the nonnegative integer dilation factor. The polynomial is called the \emph{Ehrhart tensor polynomial}. Equivalently, if $P$ is a $d$-dimensional lattice polytope,
\[
\sum _{n\geq 0}L^r(nP)t^n \ = \ \frac{h_0^r(P)+h_1^r(P)t+\cdots + h_{d+r}^r(P)t^{r+d}}{(1-t)^{d+r+1}}
\]
for tensors $h_0^r(P),h_1^r(P),\ldots, h_{r+d}^r(P)\in \Te^r$. The numerator polynomial is called the \emph{$h^r$-tensor polynomial} of $P$~\cite{Berg}. 
Observe that for $r=0$ we recover the usual Ehrhart and $h^\ast$-polynomial 
of a lattice polytope.

The vector space of symmetric tensors $\Te ^r$ is isomorphic to the vector 
space of multi-linear functionals $(\R^d)^r\rightarrow \R$ that are invariant under permutations of the arguments. In particular, for any $v_1,\ldots, v_r\in \R^d$,
\[
L^r(P)(v_1,\ldots, v_r) \ = \ \sum _{x\in P\cap \Z^d}(x^T v_1)\cdots (x^T v_r) \, .
\]
Thus, weighted Ehrhart polynomials can be seen as evaluations of Ehrhart tensor polynomials in the following sense.
\begin{prop}\label{prop:identification}
Let $w(x)=\ell _1 (x)\cdots \ell _r(x)$ be a product of linear forms where each 
linear form $\ell _i : \R^d \rightarrow \R$ is given by $\ell _i (x)=x^T v_i$ 
for some $v_i\in \R^d$. Let $P$ be a $d$-dimensional lattice polytope. Then
\[
\ehr(nP,w) \ = \ \sum _{i=0}^{d+r}L^r(P)(v_1,\ldots, v_r)n^i 
\]
and, equivalently,
\[
h^\ast _{P,w}(t) \ = \ \sum _{i=0}^{d+r}h_i^r (P) (v_1,\ldots, v_r)t^i \, .
\]
\end{prop}
\begin{proof}
    For any integer $n\geq 0$,
    \[
    \ehr(nP,w) =\sum _{x\in nP\cap Z^d}x^Tv_1 \cdots x^Tv_r= L^r(nP)(v_1,\ldots, v_r)=\sum _{i=0}^{d+r}L_i^r(nP)(v_1,\ldots, v_r)n^i \, .
    \]
    The claim for the $h^\ast$-polynomials follows similarly.
\end{proof}
In the case that $r=2$, symmetric tensors can be identified with symmetric matrices via their values on pairs of standard vectors. Via this identification, a tensor is called positive semi-definite if the corresponding matrix is positive semi-definite. In particular, $L^2(P)=\sum _{x\in P\cap Z^d}x x^T$ is always positive semi-definite. However, the coefficients of the Ehrhart tensor polynomial and the $h^2$-tensor polynomial need not be in general~\cite{Berg}, similarly as the coefficients of the usual Ehrhart polynomial are not positive in general. The following relation between the positivity of weighted $h^\ast$-polynomials and the positive semi-definiteness of the coefficients of the $h^2$-tensor polynomial is a consequence of Proposition~\ref{prop:identification}.

\begin{prop}\label{prop:positivesemidefinite}
    For any lattice polytope $P\subset \R^d$, the $h^2$-tensor polynomial of $P$ has only positive semi-definite coefficients if and only if $h^\ast _{P,w}(t)$ has only nonnegative coefficients for each weight that is a square of a linear form $w(x)=\ell^2 (x)$.
\end{prop}
\begin{proof}
    Let $M_i=h_i^2(P)\in \R^{2\times 2}$ be the coefficients of the $h^2$-polynomial of $P$. By Proposition~\ref{prop:identification}, for any linear form $\ell (x)=v^Tx$ on $\R^d$, $h^\ast _{P,w}(t)=\sum _{i} v^TM_iv t^i$. Thus, $h^\ast _{P,w}(t)$ has only nonnegative coefficients for all weights $w(x)=\ell (x)^2$ if and only if the matrices $M_i$ are all positive semi-definite.
\end{proof}

In~\cite{Berg} Berg, Jochemko and Silverstein investigated when $h^2$-tensor polynomials have only positive semi-definite coefficients. They proved that the coefficients are indeed positive semi-definite for lattice polygons~\cite[Theorem 5.2]{Berg} and conjectured that this holds more general in arbitrary dimensions~\cite[Conjecture 6.1]{Berg}. By Proposition~\ref{prop:positivesemidefinite}, it follows that Theorem~\ref{thm:polygon} is equivalent to~\cite[Theorem 5.2]{Berg}; the proof given in Section~\ref{sec:2d} is arguably simpler.
\begin{cor}[{\cite[Theorem 5.2]{Berg}}]
The $h^2$-tensor polynomial of any lattice polygon has only positive semi-definite coefficients.
\end{cor}

Furthermore,  Example~\ref{ex:negative} provides a $20$-dimensional lattice polytope together with a weight $w(x)=\ell (x)^2$ that is a square of a linear form such that $h_{P,w}(t)$ has a negative coefficient. By Proposition~\ref{prop:positivesemidefinite} this establishes a counterexample to ~\cite[Conjecture 6.1]{Berg}.

\begin{cor}\label{cor:tensorcounterexample}
There exists a $20$-dimensional lattice polytope whose $h^2$-tensor polynomial has a coefficient that is not positive semi-definite. In particular, this disproves~\cite[Conjecture 6.1]{Berg}
\end{cor}

\section{Open question}
In Theorem~\ref{thm:main} we have proved sufficient conditions on the homogeneous weight function that yield nonnegative coefficients of the $h^\ast$-polynomial. We also have shown our results are tight, in particular, in Section \ref{sec:negative} we have seen that Theorem~\ref{thm:main} can fail if the assumptions are relaxed, even in the simple case of a  square of a single linear form.

We end this article posing a natural question.

\begin{ques}
Can we precisely characterize the family of homogeneous weights
  that yield nonnegative coefficients of the $h^*$-polynomial? 
\end{ques}

\section*{Acknowledgements}
We are grateful to the anonymous referees for carefully reading our manuscript. We thank the American Institute of Mathematics and the organizers of the  workshop ``Ehrhart polynomials: inequalities and extremal constructions'' in May 2022, as this project was started during that workshop. 
JDL was partially supported by NSF grant \#1818969 and ICERM.
AG was partially supported by the Alexander von Humboldt Foundation. 
SGM was supported by the Deutsche Forschungsgemeinschaft (DFG, German Research Foundation) under Germany's Excellence Strategy – The Berlin Mathematics Research Center MATH+ (EXC-2046/1, project ID: 390685689) and the Graduiertenkolleg ``Facets of Complexity'' (GRK 2434).
KJ was partially supported by the Wallenberg AI, Autonomous Systems and Software Program funded by the Knut and Alice Wallenberg Foundation, grant 2018-03968 of the Swedish Research Council and the Göran Gustafsson Foundation.
JY was partially supported by NSF grant \#1855726.   
\bibliographystyle{amsalpha}

\bibliography{biblio}

\end{document}